\documentclass[aos]{imsart}
\RequirePackage
{natbib}

\usepackage{amsthm,amsmath,amssymb}

\startlocaldefs
\def \mx#1{\mbox{\boldmath ${\bf #1}$}}
\newcommand{\tr}{\mathop{\mathrm{tr}}\nolimits}
\newcommand{\belle}{beautiful}
\newcommand{\gf}{\mathop{\mathrm{GF}}\nolimits}

\newtheorem{pr}{Proposition}
\newtheorem{definition}{Definition}

\newenvironment{defn}{\definition\normalfont}{\enddefinition}

\newtheorem{thm}{Theorem}
\newtheorem{lm}{Lemma}
\newtheorem{construction}{Construction}
\newenvironment{cons}{\construction\normalfont}{\endconstruction}
\newtheorem{example}{Example}
\newenvironment{eg}{\begin{example}\normalfont}{\end{example}}

\endlocaldefs

\begin{document}

\begin{frontmatter}

\title{On optimality and construction of circular repeated-measurements designs}
\runtitle{Circular weakly balanced repeated-measurements}

\begin{aug}
  \author{\fnms{R. A.}  \snm{Bailey}\thanksref{m1,m2}\ead[label=e1]{rab24@st-andrews.ac.uk}},
  \author{\fnms{Peter J.} \snm{Cameron}\thanksref{m1,m2}\ead[label=e2]{pjc20@st-andrews.ac.uk}},
  \author{\fnms{Katarzyna} \snm{Filipiak}\thanksref{m3}\corref{}\ead[label=e3]{kasfil@up.poznan.pl}},
  \author{\fnms{Joachim} \snm{Kunert}\thanksref{m4}\ead[label=e4]{kunert@statistik.tu-dortmund.de}}
  \and
  \author{\fnms{Augustyn}  \snm{Markiewicz}\thanksref{m3}\ead[label=e5]{amark@up.poznan.pl}}%


  \runauthor{R. A. Bailey et al.}

  \affiliation{\thanksmark[1]{m1}School of Mathematics and Statistics, University of St Andrews, UK\\
  \thanksmark[2]{m2}School of Mathematical Sciences, Queen Mary, University of London, UK\\
  \thanksmark[3]{m3}Department of Mathematical and Statistical Methods, Pozna\'n University of Life Sciences, Poland\\ 
  \thanksmark[4]{m4}Department of Statistics, TU Dortmund University, Germany}

\end{aug}



\begin{abstract}
The aim of this paper is to characterize and construct universally optimal designs among the class of circular
repeated-measurements designs when the parameters do not permit balance for carry-over effects.
It is shown that some circular weakly neighbour balanced designs defined by \cite{FM12} 
are universally optimal repeated-measurements designs. 
These results extend the work of \cite{Magda}, \cite{Kunert84b} and \cite{FM12}. 
\end{abstract}

\begin{keyword}[class=MSC]
\kwd[Primary ]{62K05}
\kwd[; secondary ]{62K10}
\end{keyword}

\begin{keyword}
\kwd{circular weakly balanced design}
\kwd{repeated-measurements design}
\kwd{uniform design}
\kwd{universal optimality}
\end{keyword}

\end{frontmatter}

\section{Introduction}

The problem of universal optimality of repeated-\linebreak
measurements designs is widely studied in the literature.
Most of the designs considered have the same number of periods as treatments;
we also make this assumption.

For experiments without a pre-period, \cite{Hedayat} and  \cite{Cheng}
proved the universal optimality, for the estimation of direct as well as carry-over effects, 
of some balanced uniform repeated-measurements designs over a restricted class of competing designs.
If the number of subjects is at most twice the number of treatments,
\cite{Kunert84a} showed that, for the estimation of direct effects,
balanced uniform designs are universally optimal over the class of all designs.
He also proved that if the number of experimental subjects is sufficiently large then 
a balanced uniform design is no longer optimal. Moreover, this design is not universally optimal 
for the estimation of carry-over effects when certain other special designs exist. 
\cite{Stufken} constructed some universally optimal designs using orthogonal arrays of type I.
\cite{Jones} proved universal optimality of some balanced uniform designs under the
model with random carry-over effects.

\cite{Kunert83} considered the model for repeated-measurements designs with or without 
a pre-period. He proved the universal optimality of some special generalized latin squares and generalized Youden
designs over particular classes of designs. 

A repeated-measurements design is called \emph{circular} if there is a pre-period and, for each subject,
the treatment on the pre-period is the same as the treatment on the last period.
\cite{Magda} proved the universal optimality of circular strongly balanced uniform designs and
circular balanced uniform designs over appropriate subclasses of possible designs.
\cite{Kunert84b} strengthened the results of \cite{Magda} by showing the universal optimality of 
circular balanced designs over the set of all designs.

Universal optimality of some balanced circular designs is also studied assuming a model
of repeated measurements designs in which period effects are not significant. This simpler model,
in which carry-over effects play the role of left-neighbour effects,
is known in the literature as an interference model.
\cite{Druilhet} considered optimality of circular balanced designs (CBDs) for the estimation of 
direct as well as carry-over effects, while 
\cite{BaileyDruilhet} proved their optimality for the estimation of total effects.
\cite{FM12} showed universal optimality of circular weakly balanced
designs (CWBDs) for the estimation of direct effects only.

In this paper we consider circular repeated-measurements designs under the full model and under 
two simpler models. We show universal optimality, for the estimation of direct as well as carry-over effects, of
circular weakly balanced designs, CWBDs, and we give methods of constructing some of them.
For particular parameter sets, there exists a CWBD using fewer subjects than the circular balanced uniform designs 
whose universal optimality was proved by \cite{Magda} and \cite{Kunert84b}. The idea of 
the possible reduction of number of subjects follows from the universal optimality of circular weakly neighbour balanced 
designs under the interference model with left-neighbour effects, which is proved in \cite{FM12}.

\section{Models and designs}

Let $\mathcal D_{t,n,t}$ be the set of circular designs with $t$ treatments, $n$ experimental subjects and 
$t$ periods, each subject being given one treatment during each period. By $d(i,j)$, for
$1\leq i\leq t$ and $1\leq j\leq n$, we denote the treatment assigned to the 
$j$th subject in the $i$th period. \cite{Magda} proposed the following model 
associated with the design $d$ in $\mathcal D_{t,n,t}$:
\begin{equation}
\label{model_elements}
y_{dij}=\alpha_i+\beta_j+\tau_{d(i,j)} +\rho_{d(i-1,j)} +\varepsilon_{ij} , \quad
1\leq i\leq t, 
\; 1\leq j\leq n,
\end{equation}
where $y_{dij}$ is the response of the $j$th subject in the $i$th period,  and $\alpha _i$, $\beta _j$,
$\tau _{d(i,j)}$, $\rho _{d(i-1,j)}$ are, respectively, the $i$th period effect, the $j$th subject effect,
the direct effect of treatment $d(i,j)$ and the carry-over effect of treatment $d(i-1,j)$,  
where $d(0,j)=d(t,j)$. The $\varepsilon_{ij}$ are uncorrelated random variables with common variance and zero mean.

In vector notation model (\ref{model_elements}) can be rewritten as
\begin{equation}
\label{model}
\mathbf{y}=\mathbf{P}\boldsymbol{\alpha}+\mathbf{U}{\boldsymbol\beta}+\mathbf{T}_d\boldsymbol{\tau} +
\mathbf{F}_d\boldsymbol{\rho} +\boldsymbol{\varepsilon }.
\end{equation}
Here $\mathbf{y}$ is the transpose of the vector $\mathbf{y}'=(y_{d11}, y_{d21},\dots, y_{dtn})$.
Also, $\boldsymbol{\alpha}$, $\boldsymbol{\beta}$, $\boldsymbol{\tau}$ and $\boldsymbol{\rho}$ are the vectors of period, experimental subject, 
direct and carry-over effects respectively. Moreover, $\boldsymbol{\varepsilon}$ is the vector of random errors, 
with $\boldsymbol{\varepsilon}\sim N(\mathbf{0}_{nt},\sigma^2\mathbf{I}_{nt})$, where $\sigma^2$ is a positive constant,
$\mathbf{I}_n$ denotes the identity matrix of order $n$, and $\mathbf{0}_n$ is the $n$-dimensional vector of zeros.
The matrices $\mathbf T_d$ and $\mathbf F_d$ are the design matrices for  direct and carry-over effects respectively,
while $\mathbf{P} =  \mathbf{1}_n \otimes \mathbf{I}_t$ and $\mathbf{U} = \mathbf{I}_n \otimes \mathbf{1}_t$ are
the incidence matrices for period and experimental subject effects respectively, 
where $\mathbf{1}_n$ is the $n$-dimensional vector of ones and $\otimes $ denotes the Kronecker product.
Let $\mathbf{H}_t=(h_{ij})$ be the circulant matrix of order $t$
with $h_{ij}=1$ if $j-i=1$ or $i=1, j=t$, and $h_{ij}=0$ otherwise.
Then $\mathbf{F}_d = (\mathbf{I}_n\otimes \mathbf{H}_t)\mathbf{T}_d$.

\medskip
In this paper we also consider the following simpler models -- model (\ref{model}) without period effects:
\begin{equation}
\label{m_block}
\mathbf{y}=\mathbf{U}\boldsymbol{\beta}+\mathbf{T}_d\boldsymbol{\tau} +\mathbf{F}_d\boldsymbol{\rho} +\boldsymbol{\varepsilon },
\end{equation}
and model (\ref{model}) without experimental subject effects:
\begin{equation}
\label{m_period}
\mathbf{y}=\mathbf{P}\boldsymbol{\alpha}+\mathbf{T}_d\boldsymbol{\tau} +\mathbf{F}_d\boldsymbol{\rho} +\boldsymbol{\varepsilon }.
\end{equation}
Model (\ref{m_block}) is also known in the literature as the interference model with left-neighbour effects,
where left-neighbour effects play the role of carry-over effects; 
cf.~\cite{Druilhet}, \cite{FM12}.

\medskip
Following \cite{Magda}, we give the following definitions. A design 
$d$ in $\mathcal D_{t,n,t}$ is called:
\begin{itemize}
\item[(i)] {\it uniform on the periods} if each treatment occurs the same number of times in each period;
\item[(ii)] {\it uniform on the subjects} if each treatment is assigned exactly once to each subject;
\item[(iii)] {\it uniform} if it is uniform on both periods and subjects;
\item[(iv)] {\it circular strongly balanced} (CSBD) if the collection of ordered pairs
$(d(i,j),d(i+1,j))$,  for $0\leq i\leq t-1$ and $1\leq j\leq n$, contains each ordered pair
of treatments (distinct or not) $\lambda_0$ times, where $\lambda_0 = n/t$;
\item[(v)] {\it circular balanced} (CBD) if the collection of ordered pairs
$(d(i,j),d(i+1,j))$,  for $0\leq i\leq t-1$ and $1\leq j\leq n$, contains each ordered pair
of distinct treatments $\lambda_1$ times, where $\lambda_1 = n/(t-1)$, and does not contain
any pair of equal treatments.
\end{itemize}

We additionally define  circular weakly balanced designs.
Let $\mathbf{S}_d=\mathbf{T}'_d\mathbf{F}_d=(s_{d,ij})_{1\leq i,j\leq t}$. The entry 
$s_{d,ij}$ is the number of appearances of treatment $i$ preceded by treatment $j$ in the design $d$. 
Thus the rows and columns of $\mathbf{S}_d$ sum to the vector of treatment replications.
\cite{Filipiaketal} called the matrix $\mathbf{S}_d$ the \textit{left-neighbouring matrix}.
When the number of treatments is equal to the number of periods,
\cite{FM12} called a design $d$ in $\mathcal D_{t,n,t}$
\begin{itemize}
\item[(vi)] {\it circular weakly balanced} (CWBD) if the collection of ordered pairs
$(d(i,j),d(i+1,j))$,  for 
$0\leq i\leq t-1$ and $ 1\leq j\leq n$, contains each ordered pair
of distinct treatments $\lambda$ or $\lambda-1$ times, where $\lambda = \lceil n/(t-1) \rceil$, 
in such a way that:
\begin{itemize} 
\item[(a)] $\mathbf{S}_d\mathbf{1}_t=\mathbf{S}'_d\mathbf{1}_t=n\mathbf{1}_t$, so that each treatment has replication~$n$;
\item[(b)] $\mathbf{S}_d\mathbf{S}'_d$ is completely symmetric, i.e., all diagonal entries are equal and all
off-diagonal entries are equal.
\end{itemize}
\end{itemize}
In the above definition $\lceil x\rceil$ is the smallest integer greater than or equal~to~$x$.

Note that \cite{Wilkinson} defined partially neighbour balanced designs
as designs with $s_{d,ij}\in\left\{0, \:1\right\}$ if $i\not=j$;
however, their designs are not circular, and they consider neighbours 
in more than one direction.
Some methods of constructing partially neighbour-balanced circular designs are given by
\cite{Azais}.

If $d$ is a CSBD then $\mathbf{S}_d = \lambda_0\mathbf{J}_t$, where $\mathbf{J}_t = \mathbf{1}_t\mathbf{1}_t'$;
if $d$ is a CBD then $\mathbf{S}_d = \lambda_1(\mathbf{J}_t - \mathbf{I}_t)$. 
If $d$ is a CWBD but not a CBD then $\mathbf{S}_d$ is not completely symmetric but
$\mathbf{S}_d\mathbf{S}_d'$ is.


\section{Existence conditions}
\label{sec:exist}

A necessary condition for the existence of a CBD with $t$ periods is that $(t-1)$ divides $n$: see e.g.~\cite{Druilhet},
while for the existence of a CWBD the expression $n(n-2\lambda+1)$ must be divisible by $t-1$; cf.~\cite{FM12}.
The design parameters satisfying the necessary condition for the existence of a CWBD with $t\leq 19$ and $n<3(t-1)$ are listed in 
Table~1 of \cite{FM12}.

Suppose that $d$~is a CWBD in $\mathcal{D}_{t,n,t}$ which is not a CBD.
Then $\lambda = \lceil n/(t-1)\rceil$. Put 
\begin{equation}
k=n-(\lambda-1)(t-1)
\label{eq:kdef}
\end{equation}   
in this section and Section~\ref{constr}.
Since $d$~is not a CBD, $1\leq k\leq t-2$.  Using this notation, the necessary condition for a CWBD given by \cite{FM12}  is
\begin{equation}
t-1 \quad\mbox{divides} \quad k(k-2\lambda+1).
\label{eq:fmcond}
\end{equation}

If $n=1$ and every treatment occurs once in~$d$ then $\mathbf{S}_d$ is a permutation matrix with zero diagonal and so $\mathbf{S}_d\mathbf{S}_d' = \mathbf{I}_t$.  Hence $d$~is a CWBD.
However, the design is disconnected in the sense that direct effects of treatments are completely confounded with carry-over effects, and so neither can be estimated.  
\cite{FR09} showed that a design with both direct and left-neighbour effects of treatments cannot be connected if $n=1$ or if $t$ is even and $n=2$.
If $d$~is disconnected then it cannot be considered to be universally optimal; 
in fact, the proof of Theorem 3.1 of \cite{FM12} breaks down in this case.  From now on, we assume that $d$~is connected: in  particular, $n>1$.

Let $\mathbf{A}_d = \mathbf{S}'_d - (\lambda-1)(\mathbf{J}_t - \mathbf{I}_t)$.  
Then $\mathbf{A}_d$ is a $t \times t$ matrix whose diagonal entries are all zero
and whose other entries are all in $\{0,1\}$.
Moreover, each row and column of $\mathbf{A}_d$ has $k$~non-zero entries.  Hence
$\mathbf{A}_d \mathbf{J}_t = \mathbf{J}_t \mathbf{A}_d = k\mathbf{J}_t$. Therefore
\begin{eqnarray*}
\mathbf{S}_d \mathbf{S}_d' & = &
\left[(\lambda-1)(\mathbf{J}_t - \mathbf{I}_t) + \mathbf{A}'_d \right]
\left[(\lambda-1)(\mathbf{J}_t - \mathbf{I}_t) + \mathbf{A}_d \right]\\
& = & (\lambda -1)^2\left[ (t-2)\mathbf{J}_t + \mathbf{I}_t\right] 
+2(\lambda -1)k \mathbf{J}_t + \mathbf{A}_d' \mathbf{A}_d
-(\lambda-1)(\mathbf{A}_d + \mathbf{A}_d').
\end{eqnarray*}
Thus $\mathbf{S}_d \mathbf{S}_d'$ is completely symmetric if and only if
\begin{equation}
\mathbf{A}_d' \mathbf{A}_d - (\lambda -1)(\mathbf{A}_d + \mathbf{A}_d') 
\quad\mbox{is completely symmetric}.
\label{eq:AAstuff}
\end{equation}

We shall say that design~$d$ has
\begin{description}
\item[Type I]
if $\mathbf{A}_d + \mathbf{A}_d'$ is completely symmetric;
\item[Type II]
if $\mathbf{A}_d + \mathbf{A}_d'$ is not completely symmetric and $\lambda=1$;
\item[Type III]
if $\mathbf{A}_d + \mathbf{A}_d'$ is not completely symmetric and $\lambda>1$.
\end{description}
If $d$~has Type~I or~II then $\mathbf{A}_d'\mathbf{A}_d$ is completely symmetric.
The off-diagonal entries in each row of $\mathbf{A}_d' \mathbf{A}_d$ sum to $k(k-1)$, so in this case $k(k-1)$ is divisible by $t-1$.
If $d$~has Type~I then $k=(t-1)/2$ and 
$\mathbf{A}_d + \mathbf{A}_d' = \mathbf{J}_t - \mathbf{I}_t$.  Then $t-1$ divides
$(t-1)(t-3)/4$, and so $t \equiv 3 \bmod 4$.

\begin{thm}
\label{addCBD}
Suppose that $d$ is a CWBD in $\mathcal{D}_{t,n,t}$ and $d'$ is a 
CBD in $\mathcal{D}_{t,m,t}$, for some values of $n$ and~$m$.  
Then the design~$d''$ in $\mathcal{D}_{t,n+m,t}$ which juxtaposes $d$ and~$d'$ 
is a CWBD if and only if $d$~has Type~I.
\end{thm}

\begin{proof}
If $d'$ is a CBD then $m$~is a multiple of~$t-1$ and $\mathbf{S}_{d'}$ is completely symmetric.  Hence $\mathbf{S}_{d''} = \mathbf{S}_d+\mathbf{S}_{d'}$ and so 
$\mathbf{A}_{d''} = \mathbf{A}_{d}$.
Put $\lambda' = m/(t-1)$. Condition~(\ref{eq:AAstuff}) for design~$d''$ says that
\begin{equation}
\mathbf{A}_d' \mathbf{A}_d - 
(\lambda' + \lambda -1)(\mathbf{A}_d + \mathbf{A}_d') 
\quad\mbox{is completely symmetric}.
\label{eq:AAstuffmore}
\end{equation}
If $d$~has Type~I then $\mathbf{A}_d' \mathbf{A}_d$ and $\mathbf{A}_d + \mathbf{A}_d'$ are both completely symmetric, and so condition~(\ref{eq:AAstuffmore}) is satisfied and $d''$ is a CWBD.  Conversely, if $d''$ is a CWBD then
conditions~(\ref{eq:AAstuff}) and~(\ref{eq:AAstuffmore}) are both satisfied.   
Hence $\mathbf{A}_d + \mathbf{A}_d'$ is  completely symmetric and so $d$~has Type~I.
\end{proof}

\begin{lm}
\label{lm:bound}
Suppose that $d$~is a CWBD of Type~III in $\mathcal{D}_{t,n,t}$.
\begin{itemize}
\item[(a)]
If $k = (t-1)/2$ then $\lambda \leq (k+1)/2$.
\item[(b)]
If $k < (t-1)/2$ then $\lambda \leq k$.
\item[(c)]
If $k >(t-1)/2$ then $\lambda \leq t-k$.
\end{itemize}
\end{lm}

\begin{proof}
All entries in $\mathbf{A}_d' \mathbf{A}_d$ are in the interval $[0,k]$.
If $\mathbf{A}_{d,ij} =1$ then $(\mathbf{A}_d'\mathbf{A}_d)_{ij} \in [0,k-1]$.
\begin{itemize}
\item[(a)]
If $k=(t-1)/2$ but $\mathbf{A}_d + \mathbf{A}_d'$ is not completely symmetric then some off-diagonal entries of $\mathbf{A}_d + \mathbf{A}_d'$ are equal to~$2$ while 
others are equal to~$0$.  Hence the corresponding entries of
$\mathbf{A}_d' \mathbf{A}_d - (\lambda -1)(\mathbf{A}_d + \mathbf{A}_d')$ lie in
$[-2(\lambda-1), k-1-2(\lambda-1)]$ and $[0,k]$ respectively.  If these entries are equal then
$k-1-2(\lambda-1) \geq 0$.
\item[(b)]
If $k<(t-1)/2$ then some off-diagonal entries of $\mathbf{A}_d + \mathbf{A}_d'$ are equal to $1$ or~$2$ while others are equal to~$0$.  The corresponding entries of
$\mathbf{A}_d' \mathbf{A}_d - (\lambda -1)(\mathbf{A}_d + \mathbf{A}_d')$ lie in
$[-2(\lambda-1), k-1-(\lambda-1)]$ and $[0,k]$ respectively.  If these entries are equal
then $k-1-(\lambda-1) \geq 0$.
\item[(c)]
If $k>(t-1)/2$ then $k\geq t/2$ and so the entries in $\mathbf{A}_d' \mathbf{A}_d$ are 
in $[2k-t,k]$.  Some off-diagonal entries of $\mathbf{A}_d + \mathbf{A}_d'$ are equal to~$2$,
while others are equal to $0$ or~$1$. The corresponding entries of
$\mathbf{A}_d' \mathbf{A}_d - (\lambda -1)(\mathbf{A}_d + \mathbf{A}_d')$ lie in
$[2k-t-2(\lambda-1),k-1-2(\lambda-1)]$ and $[2k-t-(\lambda-1),k]$ respectively.
If these entries are equal then $k-1-2(\lambda-1) \geq 2k-t-(\lambda-1)$.
\end{itemize}
\end{proof}

\begin{thm}\label{exist}
If $d$~is a CWBD in $\mathcal{D}_{t,n,t}$ and $d$~has Type~II or~III then
$d$~is not uniform on the periods.
\end{thm}

\begin{proof}
If $d$~is uniform on the periods then $t$~divides~$n$.  If $d$~has Type~II then $n\leq t-1$, and so this is not possible.  If $t$~divides~$n$ then equation~(\ref{eq:kdef}) shows that 
$t$~divides $k-\lambda+1$. Lemma~\ref{lm:bound} shows that 
if $d$~has Type~III then $0<\lambda \leq k < t-1$ and so $0<k-\lambda+1<t$:
thus $t$~cannot divide $k-\lambda +1$.
\end{proof}


\section{Optimality}

\subsection{Preliminaries}

\cite{Kunert84b} showed that any CBD which is uniform on subjects is universally optimal 
for the estimation of direct as well as carry-over effects under the model (\ref{m_block}) over the class $\mathcal D_{t,n,t}$.
\cite{Druilhet} extended these results to designs where the number of periods is any multiple of~$t$.
\cite{FM12} defined circular weakly neighbour balanced designs 
to be CWBDs for $t$ periods which are uniform on subjects; they showed their universal optimality 
for the estimation of direct effects under the model (\ref{m_block}) over the class $\mathcal D_{t,n,t}$
with $n<t-1$ and over the class of equireplicated designs without self-neighbours if $n>t-1$.
One aim of this paper is to prove universal optimality for the estimation of direct as well as carry-over effects of
uniform CWBDs, CWBDs uniform on subjects and CWBDs uniform on periods under the model (\ref{model}), (\ref{m_block})
and (\ref{m_period}), respectively.

\medskip
We are interested in determining optimal designs, i.e.~designs with minimal (in some sense) variance of
the best linear unbiased estimator of the vector of parameters. \cite{Kiefer} formulated the universal optimality 
criterion in terms of the information matrix, which is the inverse of a variance-covariance matrix; 
cf.~\cite{Pukelsheim}. Therefore, following Proposition~1 of \cite{Kiefer}, we suppose 
that a class $\mathcal C=\{\mx C_{d}:d\in\mathcal D_{t,n,t}\}$
of non-negative definite information matrices with zero row and column sums contains a 
matrix $\mx C_{d^*}$ which is completely symmetric and has maximal trace over 
$\mathcal D_{t,n,t}$. Then the design $d^*$ is {\it universally optimal}\, in Kiefer's sense in the class 
$\mathcal D_{t,n,t}$.

For an $m_1\times m_2$ matrix $\mx M$ define $\omega ^\perp (\mx M) = \mx I_{m_1}-\mx M(\mx M'\mx M)^-\mx M=\mx I_{m_1}-\omega (\mx M)$
as the orthogonal projector onto the orthocomplement of the column space of $\mx M$,
where $(\mx M'\mx M)^-$ is a generalized inverse of $\mx M'\mx M$. Then 
the information matrix for the least squares estimate of $\mx \tau$ under model (\ref{model}) is given by
\[
\mx C_d^{(\ref{model})}=\mx T'_d\omega^\perp((\mx P:\mx U:\mx F_d))\mx T_d
\]
with zero row and column sums, where $(\mx P:\mx U:\mx F_d)$ is a block matrix;
cf.~e.g.~\cite{Kunert83,Kunert84a,Kunert84b}.
Since $\omega((\mx A:\mx B))=\omega(\mx A)+\omega(\omega^\perp(\mx A)\mx B)$,
we may rewrite the matrix $\mx C_d^{(\ref{model})}$ as
\[
\mx C_d^{(\ref{model})}=\mx T'_d\omega^\perp(\mx F_d)\mx T_d-\mx T'_d\omega(\omega^\perp(\mx F_d)(\mx P:\mx U))\mx T_d.
\]

Assuming models (\ref{m_block}) and (\ref{m_period}), the information matrices
for the least squares estimate of~$\mx \tau$ are
\[
\mx C_d^{(\ref{m_block})}=\mx T'_d\omega^\perp((\mx U:\mx F_d))\mx T_d=\mx T'_d\omega^\perp(\mx F_d)\mx T_d-\mx T'_d\omega({\omega^\perp(\mx F_d)\mx U})\mx T_d
\]
and
\[
\mx C_d^{(\ref{m_period})}=\mx T'_d\omega^\perp((\mx P:\mx F_d))\mx T_d=\mx T'_d\omega^\perp(\mx F_d)\mx T_d-\mx T'_d\omega({\omega^\perp(\mx F_d)\mx P})\mx T_d
\]
respectively.

Assuming models (\ref{model}), (\ref{m_block}) and (\ref{m_period}), the information matrices
for the least squares estimate of~$\mx \rho$ are
\[
\widetilde{\mx C}_d^{(\ref{model})}=\mx F'_d\omega^\perp((\mx P:\mx U:\mx T_d))\mx F_d=\mx F'_d\omega^\perp(\mx T_d)\mx F_d-\mx F'_d\omega(\omega^\perp(\mx T_d)(\mx P:\mx U))\mx F_d,
\]
\[
\widetilde{\mx C}_d^{(\ref{m_block})}=\mx F'_d\omega^\perp((\mx U:\mx T_d))\mx F_d=\mx F'_d\omega^\perp(\mx T_d)\mx F_d-\mx F'_d\omega(\omega^\perp(\mx T_d)\mx U)\mx F_d
\]
and
\[
\widetilde{\mx C}_d^{(\ref{m_period})}=\mx F'_d\omega^\perp((\mx P:\mx T_d))\mx F_d=\mx F'_d\omega^\perp(\mx T_d)\mx F_d-\mx F'_d\omega(\omega^\perp(\mx T_d)\mx P)\mx F_d
\]
respectively: cf.~e.g.~\cite{Kunert84b}.

As in \cite{Magda}, \cite{Kunert84a,Kunert84b} and \cite{Jones}, we assume 
that there are $t$ periods. 
Denote by $\mx \Theta_{t,n}$ the  $t\times n$ matrix of zeros.
Note that
$\mx T'_d\omega^\perp(\mx F_d)\mx U=\mx F'_d\omega^\perp (\mx T_d)\mx U=\mx\Theta_{t,n}
$
for every design uniform on subjects and 
$\mx T'_d\omega^\perp(\mx F_d)\mx P=\mx F'_d\omega^\perp(\mx T_d)\mx P=\mx \Theta_{t,t}$
for every design uniform on periods. 

For two symmetric matrices $\mx M_1$ and $\mx M_2$ of the same size
we say that  $\mx M_1$ is below $\mx M_2$ in the Loewner ordering,
and we write $\mx M_1\leq _L\mx M_2$, if $\mx M_2-\mx M_1$ is 
a non-negative definite matrix; cf.~\cite{Pukelsheim}.
The following proposition holds; cf.~\cite{Magda}, \cite{Kunert83}.

\begin{pr}
\label{pr_ineq}
Let $\mx C_d^{(g)}$ and $\widetilde{\mx C}_d^{(g)}$, for $g=2$, $3$ and $4$,
be the information matrices of the design $d$ in $\mathcal D_{t,n,t}$
for the estimation of $\mx \tau$ and $\mx \rho$ in model (g).
Then $\mx C_d^{(g)}\leq_L \mx T'_d\omega^\perp(\mx F_d)\mx T_d$ 
and $\widetilde{\mx C}_d^{(g)}\leq_L \mx F'_d\omega^\perp(\mx T_d)\mx F_d$,
with equality holding if and only if
\[
\mx T'_d\omega^\perp(\mx F_d)\mx U=\mx \Theta_{t,n} \qquad {\rm and} \qquad
\mx T'_d\omega^\perp(\mx F_d)\mx P= \mx{\Theta}_{t,t}
\]
and
\[
\mx F'_d\omega^\perp(\mx T_d)\mx U=\mx \Theta_{t,n} \qquad {\rm and} \qquad
\mx F'_d\omega^\perp(\mx T_d)\mx P=
\mx{\Theta}_{t,t}
\]
respectively.
\end{pr}

\subsection{Optimality results}

\cite{FM12} showed that for a CWBD
$\mathbf{S}_d\mathbf{S}_d' = \phi  \mathbf{I}_t + \xi \mathbf{J}_t$
with $\phi =n(2\lambda-1)-\lambda(\lambda-1)t-n(n-2\lambda+1)/(t-1)$
and $\xi =\lambda(\lambda-1)+n(n-2\lambda+1)/(t-1)$.
Since $\mathbf{S}_d$ is nonsingular and commutes with $\mathbf J_t$,
pre-multiplying by $\mathbf S'_d$ and post-multiplying by $(\mx S'_d)^{-1}$
we get $\mathbf{S}_d\mathbf{S}_d' = \mathbf{S}_d'\mathbf{S}_d$;
cf.~\cite[Theorem~5.2.1]{Raghavarao}, \cite{FM14}.
This shows the following.

\begin{pr}
\label{direct_carryover}
If $d$ is a CWBD then $d$ is universally optimal for the estimation of direct effects if and only if $d$ is optimal for
the estimation of carry-over effects, under the models~(\ref{model}), (\ref{m_block}), (\ref{m_period}).
\end{pr}

Let $\Lambda_{t,n,t}$ be the class of designs in $\mathcal D_{t,n,t}$ with no treatment preceded by itself.
Using the above Proposition we can complete Theorem~3.1 and Theorem~3.2 of \cite{FM12} 
as follows.

\begin{thm}
\label{interference}
If there exists a CWBD in $\mathcal D_{t,n,t}$ 
which is uniform on subjects, then it is universally optimal for the estimation
of left-neighbour effects under the model (\ref{m_block}) over the collection of 
designs in $\mathcal D_{t,n,t}$ if $n\leq t-1$ and over the collection of equireplicated
designs in $\Lambda_{t,n,t}$ otherwise.
\end{thm}

Observe 
that a 
necessary condition for the existence of designs which are uniform on periods
is 
that $n$ is divisible by $t$. 
Therefore, considering CWBD uniform on periods, we have to assume that 
$n>t-1$. The following theorem can be proved in the same way as Theorem~3.2 of \cite{FM12}
using additionally Proposition~{\ref{direct_carryover}} of this paper.
 
\begin{thm}
\label{periods}
Assume that $t>2$ and $n>t-1$.
If there exists a CWBD in $\Lambda_{t,n,t}$ which is uniform on periods,
then it is universally optimal for the estimation of direct as well as carry-over
effects under the model (\ref{m_period}) over the collection of equireplicated designs
in $\Lambda_{t,n,t}$.
\end{thm}





We now aim to strengthen Theorem~\ref{interference} by removing the condition that the competing designs are equireplicate when $n\geq t$.  Theorem~\ref{addCBD} shows that if $d$~is a CWBD which  is not a CBD then we can make larger CWBDs by juxtaposing $d$ with one or more CBDs only if $d$ has Type~I.  Therefore we restrict attention to the case that the parameter~$k$ in Section~\ref{sec:exist} is equal to $(t-1)/2$, where $t\equiv 3\bmod{4}$, and $n$~is an odd multiple of $(t-1)/2$.

At the same time, we aim to prove the analogous result for uniform designs under 
model~(\ref{model}).  Theorem~\ref{exist} shows that if $d$ is a uniform CWBD which is not a CBD then $t\equiv 3\bmod{4}$, $k=(t-1)/2$ and $n$~is an odd multiple of $t(t-1)/2$.  Thus we now suppose that $n=(2h+1)(t-1)/2$, where $h$ is an arbitrary integer and $2h+1\geq t$.
We give two technical lemmas and three propositions which are useful to prove our two results, which are combined in Theorem~\ref{main}.


We denote by $n_{d,iu}$ the number of times that treatment~$i$ appears in 
the $u$th subject ($\mx T'_d\mx U=\mx F'_d\mx U=(n_{d,iu})$) and by $r_{d,i}$ the number of times
that treatment $i$ appears in the design. We further define
\[
a_d=\sum_{i=1}^t\sum_{u=1}^n \max\{n_{d,iu}-1, 0\}.
\]
As shown by e.g.~\cite{Kunert84b}, if $g=2$ or $g=3$ then
\[
\tr \mx C_d^{(g)} \le \sum_{i=1}^t r_{d,i} -\frac{1}{t} \sum_{i=1}^t \sum_{u=1}^n n_{d,iu}^2 - 
\sum_{i=1}^t \sum_{j=1}^t \frac{(s_{d,ij}-\frac{1}{t} \sum_{u=1}^n n_{d,iu} n_{d,ju} )^2  } {r_{d,j}}.
\]

We will use the following technical lemmas.

\begin{lm}
\label{lemma1}
Assume that $x_1, x_2, \dots, x_k$ satisfy $\sum_{i=1}^k x_i =c$.
Then $\sum_{i=1}^k x_i^2 \ge \frac{1}{k} c^2$.
\end{lm}

\begin{proof}
Define $e_i=x_i-\frac{c}{k}$ for $1\le i \le k$. Then $\sum_{i=1}^k e_i =0$. It follows that
\[
\sum_{i=1}^k x_i^2 =k\left(\frac{c}{k}\right)^2+ \sum_{i=1}^k e_i^2 \ge \frac{c^2}{k}.
\]
\end{proof}

\begin{lm}
\label{lemma2}
Assume that $x_1, x_2, \dots, x_k$ are integers, such that $\sum_{i=1}^k x_i =c$ and 
that $c/k$ is also an integer. Define $\sum_{i=1}^k \max \{x_i-\frac{c}{k}, 0\}=a$.
Then $\sum_{i=1}^k x_i^2 \ge \frac{1}{k} c^2 + 2a$.
\end{lm}

\begin{proof}
Define $e_i$ as in the proof of Lemma~\ref{lemma1}.
Then all $e_i$ are integers and, therefore, $e_i^2 \ge |e_i |$ for all $i$. Since the sum of all positive $e_i$ is $a$
and $\sum_{i=1}^k e_i =0$, we conclude that  $\sum_{i=1}^k |e_i | =2a$. In all, we get
\[
\sum_{i=1}^k x_i^2 = k\left(\frac{c}{k}\right)^2+ \sum_{i=1}^k e_i^2 \ge \frac{c^2}{k}+2a.
\]
\end{proof}

We use the following propositions.

\begin{pr}
\label{pr_tail}
For any design $d\in \Lambda_{t,n,t}$ we have for any treatment $j$ that
\[
\sum_{i=1}^t \left(s_{d,ij}-\frac{1}t \sum_{u=1}^n n_{d,iu} n_{d,ju}\right )^2  \ge \frac{r_{d,j}^2}{t(t-1)}.
\]
\end{pr}

\begin{proof}
Since $\sum_{i=1}^t \sum_{u=1}^n n_{d,iu} =nt$, it follows from Lemma~\ref{lemma2} that
\[
\sum_{i=1}^t \sum_{u=1}^n n_{d,iu}^2  \ge nt+2a_d.
\]
Now consider an arbitrary treatment $j$. Note that for all competing designs all $s_{d,jj}=0$. Therefore,
\[
\sum_{i=1}^t (s_{d,ij}-\frac{1}t \sum_{u=1}^n n_{d,iu} n_{d,ju}) =
-\frac{1}t \sum_{u=1}^n n_{d,ju}^2  +\sum_{i \ne j} (s_{d,ij}-\frac{1}t \sum_{u=1}^n n_{d,iu} n_{d,ju} ).
\]
It follows that
\[
\sum_{i \ne j}  (s_{d,ij}-\frac{1}t \sum_{u=1}^n n_{d,iu} n_{d,ju}) = \frac{1}t \sum_{u=1}^n n_{d,ju}^2.
\]

Applying Lemma~\ref{lemma1}, we conclude that
\[
\sum_{i=1}^t (s_{d,ij}-\frac{1}t \sum_{u=1}^n n_{d,iu} n_{d,ju} )^2  =
\frac{1}{t^2}\left(\sum_{u=1}^n n_{d,ju}^2 \right)^2  + \sum_{i \ne j} (s_{d,ij}-\frac{1}t \sum_{u=1}^n n_{d,iu} n_{d,ju} )^2
\]
\[
\ge \frac{1}{t^2} \left(\sum_{u=1}^n n_{d,ju}^2\right)^2  + \frac {1}{t-1}  \frac{1}{t^2} \left(\sum_{u=1}^n  n_{d,ju}^2 \right)^2
 =\frac {1} {t(t-1)} \left(\sum_{u=1}^n n_{d,ju}^2 \right)^2  \ge \frac {r_{d,j}^2}{t(t-1)} .
\]
\end{proof}

We therefore get 
the following general bound for the trace of the information matrix which depends on $a_d$.

\begin{pr}
\label{bound}
For any design $d \in \Lambda_{t,n,t}$ and $g=2,3$ we have
\[
\tr \mx C_d^{(g)} \le n(t-1-\frac{1}{t-1})-\frac{2a_d}{t}.
\]
\end{pr}

\begin{proof}
Applying Proposition~\ref{pr_tail}, we get
\begin{eqnarray*}
\tr \mx C_d^{(g)} & \le & \displaystyle{
nt-\frac{1}{t}(nt+2a_d)-\sum_{j=1}^t \frac{1}{r_{d,j}}  \sum_{i=1}^t (s_{d,ij}-\frac{1}{t} \sum_{u=1}^n n_{d,iu} n_{d,ju} )^2} \\
& \le & \displaystyle{ nt-\frac{1}{t} (nt+2a_d )-\sum_{j=1}^t \frac {r_{d,j}}{t(t-1)} = nt-\frac {1}{t} (nt+2a_d )-\frac{n}{t-1}}\\
& = &\displaystyle{n(t-1-\frac{1}{t-1})-\frac {2a_d}{t}}.
\end{eqnarray*}
\end{proof}

The above bound is well-known; it was used by \cite{Kunert84a,Kunert84b}.
If $a_d$ is small, we get a sharper bound which is derived in the next proposition.

\begin{pr}
\label{sharp_bound}
Assume the design $d$ is such that $a_d \le (t-1)/2$. Then
\[
\tr \mx C_d^{(g)} \le n\left(t-1-\frac{1}{t-1}\right)-\frac{2a_d}{t}-(t-2a_d )\left(\frac{t-1}{4n}-\frac{2a_d}{nt}\right).
\]
\end{pr}

\begin{proof}
There are at most $2a_d$ of the $n_{d,iu}$ not equal to $1$.
Since $a_d \le (t-1)/2$, we conclude that there must be at least $t-2a_d$ treatments $j$ such that 
$n_{d,ju}=1$ for $1\le u\le n$. Define $\mathcal J_d^*$ as the set of all such treatments.

Assume without loss of generality that treatment $t$ is in $\mathcal J_d^*$. Then $r_{d,t}=n$.
Since we consider the circular neighbour structure, $\sum_{i=1}^t s_{d,it}=r_{d,t}=n$.
Without loss of generality we can assume that the treatments are labelled in such a way that
$s_{d,1t}\ge s_{d,2t} \ge \cdots \ge s_{d,t-1,t}$.

If $s_{d,(t-1)/2,t} \le \lambda -1$, then $s_{d,it}\le \lambda-1$ for all $i>(t-1)/2$. Consequently,
\[
\sum_{i=(t+1)/2}^{t-1} s_{d,it} \le \frac{t-1}{2} (\lambda-1) = \frac{n}{2}-\frac{t-1}{4}.
\]
This, however, implies that
\[
\sum_{i=1}^{(t-1)/2} s_{d,it} \ge \frac{n}{2}+\frac{t-1}{4}.
\]
If, however, $s_{d,(t-1)/2,t} \ge \lambda$, then $s_{d,it} \ge \lambda$ for all $i \le (t-1)/2$ and
\[
\sum_{i=1}^{(t-1)/2} s_{d,it} \ge \frac{t-1}{2} \lambda = \frac{n}{2}+\frac{t-1}{4},
\]
again.

Now define
\[
c=\sum_{i=1}^{(t-1)/2} (s_{d,it}-\frac{1}{t} \sum_{u=1}^n n_{d,iu} n_{d,tu}).
\]
Because $n_{d,tu}=1$ for $1\leq u\leq n$, we conclude that
\[
c=\sum_{i=1}^{(t-1)/2} (s_{d,it}-\frac{1}{t} \sum_{u=1}^n n_{d,iu}).
\]
Note that $\sum_{i=1}^{(t-1)/2} \sum_{u=1}^n n_{d,iu} \le \frac{t-1}{2} n+a_d$ by the definition of $a_d$. We therefore get
\begin{eqnarray*}
c \ge \frac{n}{2}+\frac{t-1}{4}-\frac{1}{t} \sum_{i=1}^{(t-1)/2}\sum_{u=1}^n n_{d,iu}
& \ge & \frac{n}{2}+\frac{t-1}{4}-\frac{1}{t} \left(\frac{t-1}{2} n+a_d \right)
\\ & = &\frac{n}{2t}+\frac{t-1}{4}-\frac{a_d}{t}.
\end{eqnarray*}
It follows that $c\ge n/(2t)$ because $a_d \le (t-1)/2$. Furthermore,
\begin{eqnarray*}
0&=&\displaystyle{\sum_{i=1}^t (s_{d,it}-\frac{1}{t} \sum_{u=1}^n n_{d,iu} n_{d,tu})} \\
&=&\displaystyle{-\frac{1}{t} \sum_{u=1}^n n_{d,tu}^2 +\sum_{i=1}^{(t-1)/2} (s_{d,it}-\frac{1}{t} \sum_{u=1}^n n_{d,iu} n_{d,tu})}\\
& &\displaystyle{+ \sum_{i=(t+1)/2}^{t-1} (s_{d,it}-\frac{1}{t} \sum_{u=1}^n n_{d,iu} n_{d,tu})}\\
&=&\displaystyle{-\frac{1}{t} n+c+\sum_{i=(t+1)/2}^{t-1} (s_{d,it}-\frac{1}{t} \sum_{u=1}^n n_{d,iu} n_{d,tu})}
\end{eqnarray*}
and therefore
\[
\sum_{i=(t+1)/2}^{t-1} (s_{d,it}-\frac{1}{t} \sum_{u=1}^n n_{d,iu} n_{d,tu}) = \frac{n}{t}-c.
\]
Applying Lemma~\ref{lemma1} we conclude that
\begin{eqnarray*}
\sum_{i=1}^t (s_{d,it}-\frac{1}{t} \sum_{u=1}^n n_{d,iu} n_{d,tu} )^2  
& \ge &
\frac{n^2}{t^2} +\frac{2c^2}{t-1} +\frac{2(n/t-c)^2}{t-1}
\\ & = & 
\frac{t+1}{t-1}  \frac{n^2}{t^2} +\frac{4}{t-1} (c^2-\frac{nc}t).
\end{eqnarray*}
Since $c \ge \frac{n}{2t}$, this bound is increasing in $c$. 
Inserting $c \ge n/(2t)+(t-1)/4-a_d/t$, we get
\begin{eqnarray*}
\sum_{i=1}^t (s_{d,it}-\frac{1}{t} \sum_{u=1}^n n_{d,iu} n_{d,tu} )^2 
& \ge &
\frac{n^2}{t(t-1)} +\frac{t-1}4+\frac{4a_d^2}{(t-1) t^2}-\frac{2 a_d}t 
\\ & \ge &
\frac{n^2}{t(t-1)} +\frac{t-1}4 -2 \frac{a_d}t  .
\end{eqnarray*}
For any treatment $j$ which does not have all $n_{d,ju}=1$, we use the bound
from Proposition~\ref{pr_tail}. Combining these two, we get
\begin{eqnarray*}
\tr \mx C_d^{(g)} 
& \le & 
nt-\frac{1}{t} (nt+2a_d )-\sum_{j \in \mathcal J_d^*} \frac{r_{d,j}}{t(t-1)} - 
\sum_{j \in \mathcal J_d^*} \left(\frac{n}{t(t-1)} +\frac{t-1}{4n}- 2 \frac{a_d}{nt}\right)
\\ & = & 
nt-\frac{1}{t} (nt+2a_d )-\sum_{j=1}^t \frac{r_{d,j}}{t(t-1)} - |\mathcal J_d^* |\left(\frac{t-1}{4n}-2 \frac{a_d}{nt}\right),
\end{eqnarray*}
where we have used the fact that $r_{d,j}=n$ for all $j \in \mathcal J_d^*$, and where $|\mathcal J_d^* |$  is the number of elements
of $\mathcal J_d^*$. Due to the restriction that $a_d \le (t-1)/2$, we observe that
$\frac{t-1}{4n}-2 \frac{a_d}{nt} \ge \frac{t-1}{4n} (1-4/t)>0$. Since $|\mathcal J_d^* | \ge t-2a_d$, it follows that
\[
\tr \mx C_d^{(g)} \le nt-\frac{1}{t} (nt+2a_d )-\frac{n}{t-1} -(t-2a_d)\left(\frac{t-1}{4n}-\frac{2 a_d}{nt}\right),
\]
which implies the desired inequality.
\end{proof}

Now we can prove the following theorem.

\begin{thm}
\label{main}
Assume that $t \ge 5$ and that $n \ge t(t-1)/2$. We also assume that 
$t$~is odd and that $n$~is an odd multiple of $(t-1)/2$.
If $d^*$~is a uniform CWBD in $\Lambda_{t,n,t}$ then $d^*$~is universally optimal for the estimation of direct as well as carry-over effects over the collection of 
designs in $\Lambda_{t,n,t}$ under model~(\ref{model}).
If $d^*$~is a CWBD in $\Lambda_{t,n,t}$ which is uniform on subjects then $d^*$~is universally optimal for the estimation of direct as well as carry-over effects over the collection of designs in $\Lambda_{t,n,t}$ under model~(\ref{m_block}).
\end{thm}

\begin{proof}
If the design $d$ has $a_d=0$, we get from Proposition~\ref{sharp_bound} that
\[
\tr \mx C_d^{(g)} \le n(t-1-1/(t-1))-\frac{t(t-1)}{4n},
\]
which is the trace of the information matrix of the 
CWBD~$d^*$.
Considering the simple bound derived in Proposition~\ref{bound}, we see that any design $d\in \Lambda_{t,n,t}$
can only perform better than $d^*$ if
\[
\frac{t(t-1)}{4n} \ge \frac{2a_d}{t}.
\]
Since we restrict to the case $n \ge t(t-1)/2$, the left-hand side is less than or equal to $1/2$.
If, however, $a_d>(t-1)/2$, then the right-hand side is greater than $(t-1)/t>1/2$.

Therefore, we only have to consider designs with $a_d \le (t-1)/2$   and the bound in Proposition~\ref{sharp_bound} applies.
Taking the derivative of
\[
f(a)=n(t-1-1/(t-1))-\frac{2a}t-(t-2a)\left(\frac{t-1}{4n}-\frac{2a}{nt}\right)
\]
with respect to $a$, we get
\begin{eqnarray*}
f'(a)&=&\displaystyle{-\frac{2}t+ 2\left(\frac{t-1}{4n}-\frac{2a}{nt}\right)-(t-2a)\left(-\frac{2}{nt}\right)}\\
&=&\displaystyle{-\frac{2}t+\frac{2(t-1)}{4n}+\frac{2t}{nt}-\frac{4a}{nt}-\frac{4a}{nt}}\\[.3cm]
&=&\displaystyle{\frac{-8n+2t(t-1)+8t-32a}{4nt}}\\
&\le & \displaystyle{\frac{-4t(t-1)+2t(t-1)+8t}{4nt}=\frac{-(t-1)+4}{2n}} \le 0.
\end{eqnarray*}
This, however, implies that the bound from Proposition~\ref{sharp_bound} is largest for $a_d=0$,
and for any design $d\in \Lambda_{t,n,t}$ we have
$\tr \mx C_d^{(g)} \le \tr \mx C_{d^*}^{(g)}$, 
for $g=2,3$.
\end{proof}

\section{Constructions}\label{constr}

In this section we suppose that $d$~is a CWBD in $\mathcal{D}_{t,n,t}$ which is not a CBD.
For each type of CWBD, we give constructions for a  suitable matrix~$\mathbf{A}$ and then search for a design~$d$ with $\mathbf{A}_d = \mathbf{A}$.

\subsection{Designs of Type I}
\label{sec:I}

For a design of Type~I, we have $t\equiv 3 \bmod 4$ and $k=(t-1)/2$.  
We need a $t \times t$ matrix $\mathbf{A}$ which has zero entries on the diagonal, 
$k$~entries equal to~$1$ in each row and column, and all other entries zero; 
it must also satisfy 
(a)~$\mathbf{A} + \mathbf{A}' = \mathbf{J}_t - \mathbf{I}_t$ and
(b)~$\mathbf{A}'\mathbf{A} = \phi  \mathbf{I}_t + \xi  \mathbf{J}_t$ with 
$\phi = (t+1)/4$ and $\xi  = (t-3)/4$.
Our strategy is first to find a matrix~$\mathbf{A}$ satisfying these conditions;
secondly to construct a design~$d$ which is a CWBD, is uniform on the subjects and has 
$\mathbf{A}_d = \mathbf{A}$; thirdly to construct a uniform design~$d$ which is a CWBD and has $\mathbf{A}_d = \mathbf{A}$.

The matrix~$\mathbf{A}$ can be regarded as the adjacency matrix of a 
directed graph~$\Gamma$ on $t$~vertices: there is an arc from vertex~$i$ to vertex~$j$ 
if and only if $\mathbf{A}_{ij}=1$.  This directed graph is called a \textit{doubly regular tournament} precisely when the matrix $\mathbf{A}$ satsifies the foregoing conditions: see
\cite{Reid}.
For a design which is a CWBD, is uniform on subjects and has $\lambda=1$, we need a decomposition of a doubly regular tournament~$\Gamma$ into Hamiltonian cycles.

One construction of doubly regular tournaments uses finite fields.
If $t$ is a power of an odd prime then there is a finite field $\gf(t)$ of $t$~elements.
If $t$~is prime then $\gf(t)$~is the same as~$\mathbb{Z}_t$, 
which is the ring of integers modulo~$t$. 
  Let $\mathcal{S}$ be the set of non-zero squares in $\gf(t)$, and
$\mathcal{N}$ the set of non-squares.  If $t\equiv 3 \bmod 4$ then $-1\in \mathcal{N}$;  
in this case, if we label the vertices of $\Gamma$ by the elements of $\gf(t)$ 
and define the adjacency matrix $\mathbf{A}$ by putting 
$\mathbf{A}_{ij}=1$ if and only if $j-i \in \mathcal{S}$, 
then $\Gamma$~is a doubly regular tournament: see \cite{Lidl}.
By reversing all the edges of~$\Gamma$, we obtain another doubly regular tournament, which can be made directly by using $\mathcal{N}$ in place of $\mathcal{S}$.

If $t$~is itself prime, then there is an obvious Hamiltonian decomposition of~$\Gamma$: 
the circular sequences have the form $(0,s,2s,\ldots, (t-1)s)$ for $s$~in~$\mathcal{S}$.


\begin{cons}
\label{con:Ius}
Suppose that $t\equiv 3 \bmod 4$ and $t$~is prime with $t>3$. Put $n=(t-1)/2$.  
Label the $t$~treatments and the $t$~periods by the elements of $\mathbb{Z}_t$, 
and the $n$~subjects by the elements of~$\mathcal{S}$.  
Define the design~$d$ by $d(j,s) = js$ for $j$~in~$\mathbb{Z}_t$  and
$s$~in~$\mathcal{S}$.  Then $d$~is a CWBD which is uniform on the subjects with $\lambda=1$.
\end{cons}

\begin{eg}
\label{eg:Ius7}
When $t=7$ we have $\mathcal{S} = \{1,2,4\}$.  We obtain the design in 
Figure~\ref{fig:Ius}(a), where the entries are integers modulo~$7$.
(In every figure, the rows denote periods and the columns denote subjects.)
\end{eg}

\begin{eg}
\label{eg:Ius11}
When $t=11$ we have $\mathcal{S} = \{1,3,4,5,9\}$.  This gives the design in
Figure~\ref{fig:Ius}(b), where the entries are integers modulo~$11$.
\end{eg}

\begin{figure}[htbp]
\[
\begin{array}{c@{\qquad}c@{\qquad}c}
\begin{array}[b]{ccc}
0 & 0 & 0 \\
1 & 2 & 4\\
2 & 4 & 1 \\
3 & 6 & 5\\
4 & 1 & 2\\
5 & 3 & 6\\
6 & 5 & 3
\end{array}
&
\begin{array}[b]{ccccc}
0 & 0 & 0 & 0 & 0\\
1 & 3 & 4 & 5 & 9\\
2 & 6 & 8 & 10 & 7\\
3 & 9 & 1 & 4 & 5\\
4 & 1 & 5 &9 & 3\\
5 & 4 & 9 & 3 & 1\\
6 & 7 & 2 & 8 & 10\\
7 & 10 & 6 & 2 & 8\\
8 & 2 & 10 & 7 & 6\\
9 & 5 & 3 & 1 & 4\\
10 & 8 & 7 & 6 & 2
\end{array}
&
\begin{array}[b]{ccccccc}
\infty & \infty & \infty & \infty & \infty & \infty & \infty\\
0 & 1 & 2 & 3 & 4 & 5 & 6\\
2' & 3' & 4' & 5' & 6' & 0' & 1'\\
3 & 4 & 5 & 6 & 0 & 1 & 2\\
1 & 2 & 3 & 4 & 5 & 6 & 0\\
5' & 6' & 0' & 1' & 2' & 3' & 4'\\
6' & 0' & 1' & 2' & 3' & 4' & 5'\\
1' & 2' & 3' & 4' & 5' & 6' & 0'\\
5 & 6 & 0 & 1 & 2 & 3 & 4\\
4 & 5 & 6 & 0 & 1 & 2 & 3\\
4' & 5' & 6' & 0' & 1' & 2' & 3'\\
6 & 0 & 1 & 2 & 3 & 4 & 5\\
2 & 3 & 4  & 5 & 6 & 0 & 1\\
3' & 4' & 5' & 6' & 0' & 1' & 6'\\
0' & 1' & 2' & 3' & 4' & 5' & 6'
\end{array}
\\
\mbox{(a)} & \mbox{(b)} & \mbox{(c)}
\end{array}
\]
\caption{Three CWBDs for $t$ treatments on $n$ subjects in $t$ periods which are uniform on the subjects: (a) $t=7$ and $n=3$; (b) $t=11$ and $n=5$; (c) $t=15$ and $n=7$}
\label{fig:Ius}
\end{figure}

For $n>1$, Construction~\ref{con:Ius} deals with $t=7$, $11$, $19$, $23$ and $31$ among
values of~$t$ below~$35$.

However, suitable matrices $\mathbf{A}$ also exist for many other values of~$t$.
\cite{Reid} showed that the $(t+1) \times (t+1)$ matrix
\[
\left[
\begin{array}{cc}
1 & \mathbf{1}_t'\\
\mathbf{1}_t & \mathbf{J}_t-2\mathbf{A}
\end{array}
\right]
\]
is a skew-Hadamard matrix if and only if $\mathbf{A}$ is the adjacency matrix of a doubly regular tournament.  Skew-Hadamard matrices of order $t+1$ are conjectured to exist whenever 
$t+1$ is divisible by~$4$. It is known that if $t<187$ and $t+1$ is divisible by $4$ then there is a skew-Hadamard matrix of order~$t+1$:
see \cite{Craigen}.

\cite{Reid} give the following doubling construction.  If 
$\mathbf{A}_1$ is the adjacency matrix of a doubly regular tournament~$\Gamma_1$ 
on $t$~vertices and
\begin{equation}
\mathbf{A}_2 =\left[
\begin{array}{ccc}
\mathbf{A}_1' & \mathbf{0}_t & \mathbf{A}_1 + \mathbf{I}_t\\
\mathbf{1}_t' & 0 & \mathbf{0}'_t\\
\mathbf{A}_1 & \mathbf{1}_t & \mathbf{A}_1
\end{array}
\right]
\label{eq:double}
\end{equation}
then $\mathbf{A}_2$ is the adjacency matrix of a doubly regular tournament~$\Gamma_2$ 
on $2t+1$ vertices.

\begin{eg}
\label{eg:Ius15}
Let $t=15$. Take $\Gamma_1$ to be the doubly regular tournament used in Example~\ref{eg:Ius7}.
The doubling construction~(\ref{eq:double}) gives the following matrix as the adjacency matrix of a doubly regular tournament~$\Gamma_2$ on $15$~vertices.
\begin{equation}
\left[
\begin{array}{ccccccccccccccc}
0 & 0 & 0 & 1 & 0 & 1 & 1 & 0 & 1 & 1 & 1 & 0 & 1 & 0 & 0\\
1 & 0 & 0 & 0 & 1 & 0 & 1 & 0 & 0 & 1 & 1 & 1 & 0 & 1 & 0\\
1 & 1 & 0 & 0 & 0 & 1 & 0 & 0 & 0 & 0 & 1 & 1 & 1 & 0 & 1\\
0 & 1 & 1 & 0 & 0 & 0 & 1 & 0 & 1 & 0 & 0 & 1 & 1 & 1 & 0\\
1 & 0 & 1 & 1 & 0 & 0 & 0 & 0 & 0 & 1 & 0 & 0 & 1 & 1 & 1\\
0 & 1 & 0 & 1 & 1 & 0 & 0 & 0 & 1 & 0 & 1 & 0 & 0 & 1 & 1\\
0 & 0 & 1 & 0 & 1 & 1 & 0 & 0 & 1 & 1 & 0 & 1 & 0 & 0 & 1\\
1 & 1 & 1 & 1 & 1 & 1 & 1 & 0 & 0 & 0 & 0 & 0 & 0 & 0 & 0\\
0 & 1 & 1 & 0 & 1 & 0 & 0 & 1 & 0 & 1 & 1 & 0 & 1 & 0 & 0\\
0 & 0 & 1 & 1 & 0 & 1 & 0 & 1 & 0 & 0 & 1 & 1 & 0 & 1 & 0\\
0 & 0 & 0 & 1 & 1 & 0 & 1 & 1 & 0 & 0 & 0 & 1 & 1 & 0 & 1\\
1 & 0 & 0 & 0 & 1 & 1 & 0 & 1 & 1 & 0 & 0 & 0 & 1 & 1 & 0\\
0 & 1 & 0 & 0 & 0 & 1 & 1 & 1 & 0 & 1 & 0 & 0 & 0 & 1 & 1\\
1 & 0 & 1 & 0 & 0 & 0 & 1 & 1 & 1 & 0 & 1 & 0 & 0 & 0 & 1\\
1 & 1 & 0 & 1 & 0 & 0 & 0 & 1 & 1 & 1 & 0 & 1 & 0 & 0 & 0  
\end{array}
\right]
\label{fifteen}
\end{equation}
Label the vertices, in order, $0$, $1$, $2$, $3$, $4$, $5$, $6$, $\infty$, $0'$, $1'$, $2'$, $3'$, $4'$, $5'$ and $6'$.  For $x$ in $\mathrm{GF}(7)$, there is an arc from $\infty$ to~$x$ and an arc from $x'$ to~$\infty$. 
For $x$ and $y$ in $\mathrm{GF}(7)$, there is an arc from $x$ to~$y$ if $y-x\in \mathcal{N}$; an arc from $x$ to~$y'$ if $x=y$ or $y-x\in\mathcal{S}$; 
an arc from $x'$ to~$y'$ if $y-x\in \mathcal{S}$; 
and an arc from $x'$ to~$y$ if $y-x\in\mathcal{S}$.

To find a CWBD which is uniform on subjects, we used \cite{GAP} to find a directed 
cycle~$\boldsymbol{\varphi}$ of length~$15$ starting $(\infty, 0, \ldots)$ in the doubly regular tournament~$\Gamma_2$ defined by the matrix (\ref{fifteen}) 
with the extra property that if $i$~is any non-zero element of $\mathrm{GF}(7)$ then the cycles $\boldsymbol{\varphi}$ and $\boldsymbol{\varphi}+i$ have no arc in common.  
Here we use the conventions that if $\boldsymbol{\varphi} = (\varphi_1, \ldots,\varphi_{15})$ then $\boldsymbol{\varphi} + i= (\varphi_1+i, \ldots,\varphi_{15}+i)$, where
$\infty +i=\infty$ and $x'+i = (x+i)'$ 
for $x$ and $i$ in $\mathrm{GF}(7)$.
\cite{GAP} found all such cycles. There are $120$, and they come in groups of three because 
if $\boldsymbol{\varphi}$ is such a cycle and $s\in\mathcal{S}$ then 
$s\boldsymbol{\varphi}$ is also such a cycle
(here the convention is that $s\boldsymbol{\varphi} = (s\varphi_1, \ldots,s\varphi_{15})$,
 where $s\times \infty = \infty$ and 
$s\times x' = (sx)'$ for $s$ and $x$ in $\mathrm{GF}(7)$).  
For each such cycle $\boldsymbol{\varphi}$,  the collection of cycles 
$\boldsymbol{\varphi}$, $\boldsymbol{\varphi}+1$, \ldots, $\boldsymbol{\varphi}+6$ gives a CWBD~$d$ for $15$ treatments on $7$~subjects in $15$ periods which is uniform on subjects
and for which $\mathbf{A}_d$ is the matrix~(\ref{fifteen}).
One of these is shown in Figure~\ref{fig:Ius}(c).

Alternatively, the function \texttt{FindHamiltonianCycles} in \textit{Mathematica}~9.0 can be used to find a Hamiltonian decomposition of~$\Gamma_2$. 
\end{eg}

For $t=3$, the direct use of a doubly regular tournament gives a design with $n=1$, which is disconnected. In order to obtain a CWBD uniform on periods in which direct and carry-over effects of treatments are estimable and which is not a CBD, we need to use one of the sequences $(0,1,2)$ and $(0,2,1)$ twice and the other one once.

\bigskip

If design~$d$ is made by Construction~\ref{con:Ius} then a uniform design~$d'$ 
with $t(t-1)/2$~subjects may be obtained by replacing the sequence $\boldsymbol{\varphi}$ 
for each subject  by the sequences $\boldsymbol{\varphi} +i$ for all $i$ in $\gf(t)$.  However, this has the effect that $\mathbf{S}_{d'} = t \mathbf{S}_d$, so $d'$ is not a CWBD, because the off-diagonal entries of $\mathbf{S}_{d'}$ are both $0$ and $t$.  Thus we need a different construction for uniform CWBDs.

Consider $\gf(t)$, where $t \equiv 3 \bmod 4$.  
If $x$ and $y$ are both in $\mathcal{S}$ or $\mathcal{N}$ then $xy\in\mathcal{S}$; 
if one is in $\mathcal{S}$ and the other in $\mathcal{N}$ then $xy \in\mathcal{N}$: 
see \cite{Lidl}.

If $\boldsymbol{\varphi}$ is any sequence $(\varphi_1, \ldots, \varphi_m)$ of 
elements of~$\gf(t)$, we denote by $\boldsymbol{\varphi}^\delta$ the sequence 
$(\varphi_2 - \varphi_1, \varphi_3-\varphi_2, \ldots, 
\varphi_m-\varphi_{m-1}, \varphi_1-\varphi_m)$
of successive circular differences in $\boldsymbol{\varphi}$. 
Further, let $f_0(\boldsymbol{\varphi}^\delta)$, 
$f_\mathcal{S}(\boldsymbol{\varphi}^\delta)$ and 
$f_\mathcal{N}(\boldsymbol{\varphi}^\delta)$ 
be the number of entries of $\boldsymbol{\varphi}^\delta$ which are 
in $\{0\}$, $\mathcal{S}$ and $\mathcal{N}$ respectively.

\begin{defn}
Let $\boldsymbol{\varphi}$ be a sequence of length $t$ whose entries are in~$\gf(t)$, 
where $t$ is a prime power congruent to $3$ modulo $4$. Then $\boldsymbol{\varphi}$
is \emph{\belle} if the entries in $\boldsymbol{\varphi}$ are all different and
$f_{\mathcal{S}}(\boldsymbol{\varphi}^\delta) = 
f_{\mathcal{N}}(\boldsymbol{\varphi}^\delta) \pm 1$. 
\end{defn}

If all of  the entries of $\boldsymbol{\varphi}$ are different then 
$f_0(\boldsymbol{\varphi}^\delta) = 0$.  Thus if $\boldsymbol{\varphi}$
has length~$t$ then it is \belle\ if and only if
$f_\mathcal{S}(\boldsymbol{\varphi}^\delta) \in \{k,k+1\}$.

\begin{cons}
\label{cons:Iusp}
Given a \belle\ sequence $\boldsymbol{\varphi} = (\varphi_1, \ldots, \varphi_t)$
of all the elements of $\gf(t)$, 
form the $t(t-1)/2$ sequences $s\boldsymbol{\varphi}+i$ for all $s$ in $\mathcal{S}$
and all $i$ in $\gf(t)$.
Create the design~$d$ by using each of these sequences for one subject.
\end{cons}

\begin{thm}
\label{thm:Pseq}
Suppose that $t \equiv 3 \bmod 4$ and $t$ is a prime power. 
If $\boldsymbol{\varphi}$ is \belle\ then the design~$d$ given by 
Construction~\ref{cons:Iusp} is a uniform CWBD.
\end{thm}

\begin{proof}
The entries in $\boldsymbol{\varphi}$ are all different, so the entries in
$s\boldsymbol{\varphi}+i$ are all different for each value of $s$ and $i$.  
Therefore each treatment occurs once on each subject, so $d$~is uniform on subjects and no treatment is preceded by itself.

For each fixed $s$ in $\mathcal{S}$, every element of $\gf(t)$ occurs once in each position among the $t$~sequences of the form $s\boldsymbol{\varphi} +i$, as $i$~varies in $\gf(t)$.  Therefore $d$ is uniform.

Consider period~$j$. Put $\varphi_j^\delta = v$.  First, suppose that $v\in\mathcal{S}$.
Let $i \in \gf(t)$ and $s\in\mathcal{S}$.  
Then $s/v\in \mathcal{S}$, and treatment $i$ occurs in period~$j$ of the sequence
$(s/v)\boldsymbol{\varphi} +w$, where $w = i-(s/v)\varphi_j$.  
The treatment in period $j+1$ of this sequence is
$(s/v)\varphi_{j+1} + w = (s/v)\varphi_{j+1} + i - (s/v)\varphi_j 
= i + (s/v)(\varphi_{j+1}-\varphi_j) = i+s$.
Thus every ordered pair of treatments of the the form $(i,i+s)$, 
for $i$ in $\gf(t)$ and $s$ in $ \mathcal{S}$, occurs exactly once in 
periods $j$ and $j+1$.

Similarly, if  $v\in \mathcal{N}$ and $q\in \mathcal{N}$ then
$q/v\in\mathcal{S}$, and every ordered pair of treatments of the form
$(i,i+q)$, for $i$ in $ \gf(t)$ and $q$ in $\mathcal{N}$, occurs exactly once in 
periods $j$ and $j+1$.

Thus if $w-i\in\mathcal{S}$ then $(i,w)$ occurs 
$f_\mathcal{{S}}(\boldsymbol{\varphi}^\delta)$ times in the design, 
while if $w-i\in\mathcal{N}$ then $(i,w)$ occurs 
$f_\mathcal{{N}}(\boldsymbol{\varphi}^\delta)$ times.  
If $\boldsymbol{\varphi}$ is \belle\ then the off-diagonal entries of $\mathbf{S}_d$ 
are in $\{k,k+1\}$ and $\mathbf{A}_d$ is the adjacency matrix of
one of the doubly regular tournaments defined by $\mathcal{S}$ or $\mathcal{N}$.
Hence $d$~is a CWBD.
\end{proof}

\begin{eg}
\label{eg:Iusp7}
Let $t=7$ and $\boldsymbol{\varphi} = (3,1,0,2,6,4,5)$, where the entries are the integers modulo~$7$.  Then $\boldsymbol{\varphi}^\delta = (5,6,2,4,5,1,5)$.  Here 
$\mathcal{S} = \{1,2,4\}$ and $\mathcal{N} = \{3,5,6\}$, and so  
$f_{\mathcal{S}}(\boldsymbol{\varphi}^\delta)=3$ and 
$f_{\mathcal{N}}(\boldsymbol{\varphi}^\delta) = 4$. Thus $\boldsymbol{\varphi}$ is \belle.
Hence Construction~\ref{cons:Iusp} gives the uniform CWBD 
for $7$~treatments 
on $21$ subjects in $7$~periods in Figure~\ref{fig:Iusp7}.
\end{eg}

\begin{figure}[htbp]
\[
\begin{array}{ccccccccccccccccccccc}
3 & 4 & 5 & 6 & 0 & 1 & 2 & 6 & 0 & 1 & 2 & 3 & 4 & 5 & 5 & 6 & 0 & 1 & 2 & 3 & 4\\
1 & 2 & 3 & 4 & 5 & 6 & 0 & 2 & 3 & 4 & 5 & 6 & 0 & 1 & 4 & 5 & 6 & 0 & 1 & 2 & 3\\
0 & 1 & 2 & 3 & 4 & 5 & 6 & 0 & 1 & 2 & 3 & 4 & 5 & 6 & 0 & 1 & 2 & 3 & 4 & 5 & 6\\
2 & 3 & 4 & 5 & 6 & 0 & 1 & 4 & 5 & 6 & 0 & 1 & 2 & 3 & 1 & 2 & 3 & 4 & 5 & 6 & 0\\
6 & 0 & 1 & 2 & 3 & 4 & 5 & 5 & 6 & 0 & 1 & 2 & 3 & 4 & 3 & 4 & 5 & 6 & 0 & 1 & 2\\
4 & 5 & 6 & 0 & 1 & 2 & 3 & 1 & 2 & 3 & 4 & 5 & 6 & 0 & 2 & 3 & 4 & 5 & 6 & 0 & 1\\
5 & 6 & 0 & 1 & 2 & 3 & 4 & 3 & 4 & 5 & 6 & 0 & 1 & 2 & 6 & 0 & 1 & 2 & 3 & 4 & 5\\
\end{array}
\]
\caption{A uniform CWBD for $7$ treatments on $21$ subjects in $7$ periods}
\label{fig:Iusp7}
\end{figure}

Now let $x$~be any primitive element of $\gf(t)$; that is, $x$ is a generator of the cyclic group $(\gf(t) \setminus\{0\},\times)$.  The even powers of~$x$ constitute $\mathcal{S}$, while the odd powers constitute~$\mathcal{N}$.
Let $\boldsymbol{\psi}$~be the sequence $(1,x,x^2, \ldots, x^{t-2})$.  Then 
$\boldsymbol{\psi}$~contains each non-zero element of $\gf(t)$ exactly once.  The 
entries in $\boldsymbol{\psi}^\delta$ are $x-1$, $x(x-1)$, \ldots, $x^{t-3}(x-1)$ and 
$1-x^{t-2} = x^{t-2}(x-1) = 1-x^{-1}$, which are again all the  non-zero elements of 
$\gf(t)$ exactly once, and so $f_{\mathcal{S}}(\boldsymbol{\psi}^\delta) = 
f_{\mathcal{N}}(\boldsymbol{\psi}^\delta) = k$.

\begin{thm}
\label{thm:n}
Let $t$ be a prime power congruent to $3$ modulo $4$ with $t>3$.
If  $x$~is a primitive element of~$\gf(t)$  
and $\boldsymbol{\varphi}$ is obtained from $\boldsymbol{\psi}$ by replacing 
$(1,x)$ with $(x,1,0)$, then $\boldsymbol{\varphi}$~is \belle.
\end{thm}

\begin{proof}
If $t>3$, the substitution removes $1-x^{-1}$, $x-1$ and $x^2-x$ from 
$\boldsymbol{\psi}^\delta$, and replaces them in $\boldsymbol{\varphi}^\delta$ by 
$x-x^{-1}$, $1-x$, $-1$ and $x^2$.  None of these is zero if $t>3$.
Now, $-1\in\mathcal{N}$ and $x^2\in\mathcal{S}$. Since $x\in\mathcal{N}$,
one of $x-1$ and $x(x-1)$ is in $\mathcal{S}$ and the other is in $\mathcal{N}$.
Since $1-x^{-1} = (-x^{-1})(1-x)$ and $-x^{-1}\in\mathcal{S}$, 
the entries $1-x^{-1}$ and $1-x$ are either both in $\mathcal{S}$ or both in $\mathcal{N}$.
Thus $f_{\mathcal{S}}(\boldsymbol{\varphi}^\delta) = 
f_{\mathcal{S}}(\boldsymbol{\psi}^\delta) + 1 = k+1$ if $x-x^{-1}\in\mathcal{S}$, 
while $f_{\mathcal{S}}(\boldsymbol{\varphi}^\delta) = 
f_{\mathcal{S}}(\boldsymbol{\psi}^\delta)  = k$ if $x-x^{-1} \in \mathcal{N}$.
\end{proof}

If $t=7$ then $3$ is a primitive element.  The construction in 
Theorem~\ref{thm:n} gives the \belle\ sequence $\boldsymbol{\varphi}$ in 
Example~\ref{eg:Iusp7}.

Theorems~\ref{thm:Pseq}--\ref{thm:n} show that there is a uniform CWBD 
for $t$~treatments on $t(t-1)/2$ subjects in $t$~periods 
whenever $t$~is a prime power congruent to $3$ modulo $4$ and $t>3$.
This covers $t=7$, $11$, $19$, $23$, $27$ and $31$ among values of~$t$ below~$35$.

\subsection{Designs of Type II}

For a design of Type~II, we have $n=k$, where $2\leq k \leq t-2$. Also, 
condition~(\ref{eq:fmcond}) shows that $t-1$ divides $k(k-1)$.  
We need a $t \times t$ matrix~$\mathbf{A}$
which has $k$~entries equal to~$1$ in each row and column, and all other entries zero, in such a way that $\mathbf{A}'\mathbf{A} = \phi  \mathbf{I}_t + \xi \mathbf{J}_t$ with
$\phi  = k(t-k)/(t-1)$ and $\xi = k(k-1)/(t-1)$.
The matrix~$\mathbf{A}$ can be regarded as the incidence matrix of a symmetric 
balanced incomplete-block design (BIBD)~$\Delta$: treatment~$i$ is in block $j$ if and only if $\mathbf{A}_{ij}=1$.  Given such a design~$\Delta$, Hall's Marriage Theorem \citep{Bailey,Cameron,HallP} shows that the treatments and blocks can be labelled in such a way that the diagonal entries of~$\mathbf{A}$ are all zero.  Now our strategy is to find a known BIBD~$\Delta$ of the appropriate size, label its blocks in such a way that the diagonal entries of $\mathbf{A}$ are all zero, and then try to find a CWBD which is uniform on subjects for which $\lambda=1$ and $\mathbf{S}_d' = \mathbf{A}_d = \mathbf{A}$.

The condition that $t-1$ divides $k(k-1)$ is not sufficient to guarantee the existence of a
BIBD.  The Bruck--Ryser--Chowla Theorem shows that some pairs $(t,k)$ have no BIBD:
see \cite{Cameron}.  For $t<35$, the following pairs are excluded by this theorem:
$(22,7)$, $(22,15)$, $(29,8)$, $(29,21)$, $(34,12)$ and $(34,22)$.

Some BIBDs can be constructed from difference sets: see \cite{HallM}.  
If $(\mathcal{G},+)$ is a finite Abelian group and $\mathcal{P}\subset\mathcal{G}$ then $\mathcal{P}$ is called a \emph{difference set} if every non-zero element of $\mathcal{G}$ occurs equally often among the differences 
$x-y$ for $x$ and $y$ in $\mathcal{P}$ with $x\ne y$.  When $\mathcal{G}$ is the additive group of $\gf(t)$ and $t\equiv 3 \bmod 4$ then $\mathcal{S}$ and $\mathcal{N}$ are both difference sets.  
If $\mathcal{P}$ is a difference set then so is its complement $\overline{\mathcal{P}}$, 
and so is the set $\mathcal{P}+i = \{x+i:i\in\mathcal{P}\}$, for each $i$ in $\mathcal{G}$.
If $i\notin\mathcal{P}$ then $\mathcal{P}-i$ is a difference set 
that does not contain~$0$. 
In particular, when $t$~is a prime power and $t\equiv 3 \bmod 4$ then 
$\overline{\mathcal{S}}-1$ is a difference set with $k=(t+1)/2$ which does not contain~$0$.

To obtain a BIBD $\Delta$ from the difference set~$\mathcal{P}$, label the treatments and blocks by the elements of~$\mathcal{G}$, and put $\mathbf{A}_{ij}=1$ if and only if 
$j-i \in \mathcal{P}$.  As above, we can assume that $0\notin\mathcal{P}$: then the diagonal entries of~$\mathbf{A}$ are all zero.

Difference sets give the following generalization of Construction~\ref{con:Ius}.

\begin{cons}
\label{cons:IIusds}
Suppose that $\mathcal{P}$ is a difference set of size~$k$ in~$\mathbb{Z}_t$, 
that $0\notin\mathcal{P}$ and that all elements of~$\mathcal{P}$ are coprime to~$t$.
Label the $t$~treatments and the $t$~periods by the elements of~$\mathbb{Z}_t$, and the
$k$~subjects by the elements of~$\mathcal{P}$. Define the design~$d$ by $d(i,j)=ij$ for 
$i$~in~$\mathbb{Z}_t$ and $j$~in~$\mathcal{P}$.  Then $d$~is a CWBD which is uniform on subjects with $\lambda=1$.
\end{cons}

\begin{eg}
\label{eg:IIusds7}
When $t=7$ and $k=4$ we have the difference set $\overline{\mathcal{S}}-1 = \{2,4,5,6\}$. Then Construction~\ref{cons:IIusds} gives the design in Figure~\ref{fig:IIusds}(a).
\end{eg}

\begin{figure}[htbp]
\[
\begin{array}{c@{\qquad}c@{\qquad}c}
\begin{array}[b]{cccc}
0 & 0 & 0 & 0\\
2 & 4 & 5 & 6\\
4 & 1 & 3 & 5\\
6 & 5 & 1 & 4\\
1 & 2 & 6 & 3\\
3 & 6 & 4 & 2\\
5 & 3 & 2 & 1
\end{array}
&
\begin{array}[b]{cccc}
0 & 0 & 0 & 0\\
1 & 2 & 5 & 7\\
2 & 4 & 10 & 1\\
3 & 6 & 2 & 8\\
4 & 8 & 7 & 2\\
5 & 10 & 12 & 9\\
6 & 12 & 4 & 3\\
7 & 1 & 9 & 10\\
8 & 3 & 1 & 4\\
9 & 5 & 6 & 11\\
10 & 7 & 11 & 5\\
11 & 9 & 3 & 12\\
12 & 11 & 8 & 6
\end{array}
&
\begin{array}[b]{ccccccccc}
0 & 0 & 0 & 0 & 0 & 0 & 0 & 0 & 0 \\
2 & 3 & 5 & 7 & 8 & 9 & 10 & 11 & 12\\
4 & 6 & 10 & 1 & 3 & 5 & 7 & 9 & 11\\
6 & 9 & 2 & 8 & 11 & 1 & 4 & 7 & 10\\
8 & 12 & 7 & 2 & 6 & 10 & 1 & 5 & 9\\
10 & 2 & 12 & 9 & 1 & 6 & 11 & 3 & 8\\
12 & 5 & 4 & 3 & 9 & 2 & 8 & 1 & 7\\
1 & 8 & 9 & 10 & 4 & 11 & 5 & 12 & 6\\
3 & 11 & 1 & 4 & 12 & 7 & 2 & 10 & 5\\
5 & 1 & 6 & 11 & 7 & 3 & 12 & 8 & 4\\
7 & 4 & 11 & 5 & 2 & 12 & 9 & 6 & 3\\
9 & 7 & 3 & 12 & 10 & 8 & 6 & 4 & 2\\
11 & 10 & 8 & 6 & 5 & 4 & 3 & 2 & 1
\end{array}
\\
\mbox{(a)} & \mbox{(b)} & \mbox{(c)}
\end{array}
\]
\caption{Three CWBDs for $t$ treatments on $n$ subjects in $t$ periods which are uniform on the subjects: (a) $t=7$ and $n=4$; (b) $t=13$ and $n=4$; (c) $t=13$ and $n=9$}
\label{fig:IIusds}
\end{figure}

Difference sets exist for many other values of $t$ and $k$ satisfying the necessary divisibility conditions:
see \cite{Baumert} and Table~2 of \cite{FM12}.

\begin{eg}
\label{eg:IIusds13}
When $t=13$, $\{1,2,5,7\}$ and $\{2,3,5,7,8,9,10,11,12\}$ are both difference sets in
$\mathbb{Z}_{13}$. Construction~\ref{cons:IIusds} gives the designs in 
Figures~\ref{fig:IIusds}(b) and~(c).
\end{eg}

\begin{eg}
\label{eg:IIusds31}
When $t=31$, $\{1,2,4,9,13,19\}$ is a difference set in $\mathbb{Z}_{31}$.  
Thus Construction~\ref{cons:IIusds} gives CWBDs uniform on subjects, one for $6$~subjects and one for $25$~subjects.
\end{eg}

A result of \cite{Mann} shows that there is no difference set of size $9$ or $16$ for 
$\mathbb{Z}_{25}$.  Theorems of \cite{Lander} rule out difference sets of size $k$ or $t-k$ for $\mathbb{Z}_t$ when $(t,k)$ is $(16,6)$, $(27,13)$ or $(31,10)$.
There is a difference set of size~$8$ for $\mathbb{Z}_{15}$, but its elements are not all coprime to~$15$, so Construction~\ref{cons:IIusds} cannot be used. The same problem occurs for $k=5$ and $k=16$ when $t=21$.

\bigskip
If $\mathbf{A}$~is symmetric then it can be regarded as the adjacency matrix of an
undirected graph~$\Omega$.  The condition that $\mathbf{A}'\mathbf{A}$ is completely symmetric means that every pair of distinct vertices has the same number of common neighbours. Such graphs were studied by \cite{Rudvalis}.  If such a graph~$\Omega$ has a Hamiltonian decomposition then using each cycle once in both directions gives a CWBD which is uniform on subjects.

\begin{eg}
\label{eg:IIsymm16a}
The smallest such graph is the square lattice graph $L_2(4)$, which has $16$~vertices and valency~$6$.  Every pair of distinct vertices have exactly two common neighbours.
The vertices form a $4\times 4$ grid.  There is an edge between $i$ and $j$ if $i \ne j$ but 
$i$ and $j$ are in the same row or $i$ and $j$ are in the same column.

Label the vertices row by row, so that the first row is $(1,2,3,4)$, and so on. Let $\pi$ be the following permutation of the vertices, which is an automorphism of~$L_2(4)$:
\[\pi = (2,3,4) (5,9,13) (6,11,16) (7,12,14) (8, 10, 15).\]
There is a Hamiltonian decomposition of~$L_2(4)$ which is invariant under~$\pi$.
Using each of these cycles in both directions gives the design in 
Figure~\ref{fig:IIsymm16}(a).
\end{eg}

\begin{figure}[htbp]
\[
\begin{array}{c@{\qquad}c}
\begin{array}{cccccc}
1 & 1 & 1 & 5 & 9 & 13\\
2 & 3 & 4 & 8 & 10 & 15\\
6 & 11 & 16 & 16 & 6 & 11\\
7 & 12 & 14 & 15 & 8 & 10\\
11 & 16 & 6 & 3 & 4 & 2\\
9 & 13 & 5 & 4 & 2 & 3\\
13 & 5 & 9 & 12 & 14 & 7\\
14 & 7 & 12 & 10 & 15 & 8\\
10 & 15 & 8 & 14 & 7 & 12\\
12 & 14 & 7 & 13 & 5 & 9\\
4 & 2 & 3 & 9 & 13 & 5\\
3 & 4 & 2 & 11 & 16 & 6\\
15 & 8 & 10 & 7 & 12 & 14\\
16 & 6 & 11 & 6 & 11 & 16\\
8 & 10 & 15 & 2 & 3 & 4\\
5 & 9 & 13 & 1 & 1 & 1 
\end{array}
&
\begin{array}{cccccccccc}
1 & 1 & 1 & 1 & 1 & 11 & 5 & 10 & 3 & 6\\
2 & 4 & 7 & 13 & 9 & 9 & 2 & 4 & 7 & 13\\
3 & 6 & 11 & 5 & 10 & 15 & 14& 12 & 8 & 16\\
4 & 7 & 13 & 9 & 2 & 13 & 9 & 2 & 4 & 7\\
7 & 13 & 9 & 2 & 4 & 14 & 12 & 8 & 16 & 15\\
5 & 10 & 3 & 6 & 11 & 16 & 15 & 14 & 12 & 8\\
6 & 11& 5 & 10 & 3  & 10 & 3 & 6 & 11 & 5\\
8 & 16 & 15 & 14 & 12 & 12 & 8 & 16 & 15 & 14\\
12 & 8 & 16 & 15 & 14 & 8 & 16 & 15 & 14 & 12\\
10 & 3 & 6 & 11 & 5 & 6 & 11 & 5 & 10 & 3\\
16 & 15 & 14 & 12 & 8 & 5 & 10 & 3 & 6 & 11\\
14 & 12 & 8 & 16 & 15 & 7 & 13 & 9 & 2 & 4\\
13 & 9 & 2 & 4 & 7 & 4 & 7 & 13 & 9 & 2\\
15 & 14 & 12 & 8 & 16 & 3 & 6 & 11 & 5 & 10\\
9 & 2 & 4 & 7 & 13 & 2 & 4 & 7 & 13 & 9\\
11 & 5 & 10 & 3 & 6 & 1 & 1 & 1 & 1 & 1 
\end{array}
\\
\mbox{(a)} & \mbox{(b)}
\end{array}
\]
\caption{Two CWBDs for $16$ treatments on $n$ subjects in $16$ periods which are uniform on the subjects: (a)  $n=6$; (b)  $n=10$}
\label{fig:IIsymm16}
\end{figure}

The Shrikhande graph is another graph with $16$ vertices, valency $6$ and the common-neighbour property: see \cite{Seidel}.  Using \cite{GAP},
we found that it 
has a very large number of Hamiltonian decompositions.  Each of these gives a CWBD that cannot be obtained from the one in Figure~\ref{fig:IIsymm16}(a) by renaming the treatments.

\begin{eg}
\label{eg:IIsymm16b}
The Clebsch graph $\Omega$ is another such graph with $16$ vertices: see \cite{Seidel}.  It has valency $10$, and every pair of distinct vertices have exactly $6$~common neighbours.  The vertices are the vectors of length~$5$ over $\gf(2)$ of even weight 
(equivalently, the treatments in the $2^{5-1}$ factorial design with defining contrast $ABCDE=I$); two vertices are joined if they differ in precisely two positions.  The permutation $\pi$ taking $(x_1,x_2,x_3,x_4,x_5)$ to $(x_2,x_3,x_4,x_5,x_1)$ is an automorphism of $\Omega$. 

Using \cite{GAP}, we found a very large number of Hamiltonian decompositions of $\Omega$ which are invariant under $\pi$ 
(as in Example~\ref{eg:IIsymm16a}, it is sufficient to find a single Hamiltonian cycle which has no edges in common with any of its images under powers of~$\pi$).  For any one of these decompositions, using each cycle in both directions gives the required CWBD. One is shown in Figure~\ref{fig:IIsymm16}(b), where vertex $(x_1,x_2,x_3,x_4,x_5)$ is identified as the integer $8x_1 + 4x_2 + 2x_3 + x_4 +1$. 
\end{eg}

For $t>16$, \cite{Rudvalis} shows that the smallest value of~$t$ for which there exists a graph with the common-neighbour property is $t=36$.

\subsection{Designs of Type III}
For a design of Type~III, we shall consider~$\mathbf{A}$ to be the adjacency matrix of a directed graph~$\Xi$.  Now $\lambda\ne 1$ and condition~(\ref{eq:AAstuff}) is satisfied.
However, neither $\mathbf{A}'\mathbf{A}$ nor $\mathbf{A} + \mathbf{A}'$ is completely symmetric, so at most one value of~$\lambda$ is possible for any given 
directed graph~$\Xi$.  
Again, we first look for such a matrix~$\mathbf{A}$ and then try to construct a CWBD $d$ for which $\mathbf{A}_d = \mathbf{A}$.
As in Section~\ref{sec:I}, we build larger matrices from smaller ones.

Let $\mathbf{A}_1$ be the adjacency matrix of a doubly regular 
tournament~$\Gamma$ on $r$~vertices, where $r=4q+3$.  Let $t=mr$, 
and put 
$\mathbf{A}_2= \mathbf{J}_m \otimes (\mathbf{I}_r + \mathbf{A}_1) - \mathbf{I}_t$.
Then $\mathbf{A}_2 + \mathbf{A}_2' = \mathbf{J}_m \otimes (\mathbf{J}_r + \mathbf{I}_r)
-2\mathbf{I}_t$ and $\mathbf{A}_2'\mathbf{A}_2 = 
(mq+m-1)\mathbf{J}_m\otimes (\mathbf{J}_r + \mathbf{I}_r) + \mathbf{I}_t$.
Thus $\mathbf{A}_2$ satisfies  condition~(\ref{eq:AAstuff}) with
$\lambda=m(q+1)$, $k= 2m(q+1) -1$ and $n=m^2(4q+3)(q+1) -m(3q+2)$.

\begin{eg}
\label{eg:III32}
When $q=0$ we may let $\mathbf{A}_1$ be the adjacency matrix of the doubly regular tournament defined by $\mathcal{S}$ in $\gf(3)$. Then
\[
\mathbf{A}_2 = 
\left[
\begin{array}{cccccc}
0 & 1 & 0 & 1 & 1 & 0\\
0 & 0 & 1 & 0 & 1 & 1\\
1 & 0 & 0 & 1 & 0 & 1\\
1 & 1 & 0 & 0 & 1 & 0\\
0 & 1 & 1 & 0 & 0 & 1\\
1 & 0 & 1 & 1 & 0 & 0
\end{array}
\right]
.
\]
If the treatments in Example~4.4 of \cite{FM12} are written in the order
$1$, $3$, $5$, $6$, $2$, $4$ then $\mathbf{A}_d = \mathbf{A}_2$. 
\end{eg}

This construction gives suitable matrices $\mathbf{A}$ for $(t,\lambda,k) = 
(3m,m,2m-1)$, $(7m,2m,4m-1)$, $(11m,3m,6m-1)$, \dots, but we have not so far found a way of constructing a corresponding CWBD directly from the smaller one.

\bigskip

\cite{Babai} give the following doubling construction for what they call 
an \emph{S-digraph}. Let $\mathbf{A}_1$ be the adjacency matrix of a doubly regular 
tournament~$\Gamma$ on $r$ vertices, where $r=4q+3$. Put 
\[
\mathbf{A}_2 = \left[
\begin{array}{cccc}
0 & \mathbf{1}_{r}' & 0 & \mathbf{0}'_{r}\\
\mathbf{0}_{r} & \mathbf{A}_1 & \mathbf{1}_{r} & \mathbf{A}'_1\\
0 & \mathbf{0}'_{r} & 0 & \mathbf{1}'_{r}\\
\mathbf{1}_{r} & \mathbf{A}_1' & \mathbf{0}_r & \mathbf{A}_1 
\end{array}
\right]
\]
and
$\mathbf{I}_{8q}^* = (\mathbf{J}_2 - \mathbf{I}_2) \otimes \mathbf{I}_{4q}$.
Then the S-digraph~$\Xi$ has adjacency matrix~$\mathbf{A}_2$.  Now,
$\mathbf{A}_2 + \mathbf{A}_2' = \mathbf{J}_{8q} - \mathbf{I}_{8q} - \mathbf{I}_{8q}^*$ and
$\mathbf{A}_2'\mathbf{A}_2 = (4q+3)\mathbf{I}_{8q} + (2q+1)(\mathbf{J}_{8q} 
- \mathbf{I}_{8q} - \mathbf{I}_{8q}^*)$.
Thus $\mathbf{A}_2$ satisfies condition~(\ref{eq:AAstuff})
with $t=8(q+1)$, $k=4q+3$, $\lambda=2(q+1)$ and $n=16q^2 +26q+10$.

\begin{eg}
\label{eq:IIIsdi}
Put $q=0$ and let $\mathbf{A}_1$ be the adjacency matrix of the doubly regular tournament defined by $\mathcal{S}$ in $\gf(3)$.  Then
\[
\mathbf{A}_2 = \left[
\begin{array}{cccccccc}
0 & 1 & 1 & 1 & 0 & 0 & 0 & 0\\
0 & 0 & 1 & 0 & 1 & 0 & 0 & 1\\
0 & 0 & 0 & 1 & 1 & 1 & 0 & 0\\
0 & 1 & 0 & 0 & 1 & 0 & 1 & 0\\
0 & 0 & 0 & 0 & 0 & 1 & 1 & 1\\
1 & 0 & 0 & 1 & 0 & 0 & 1 & 0\\
1 & 1 & 0 & 0 & 0 & 0 & 0 & 1\\
1 & 0 & 1 & 0 & 0 & 1 & 0 & 0 
\end{array}
\right].
\]
After relabelling of the treatments, this is the matrix $\mathbf{A}_d$ for the 
design~$d$ in Example~4.3 of \cite{FM12}.
\end{eg}

\vskip 3mm

\noindent ACKNOWLEDGMENTS
\vskip 3mm

This paper was started in the Isaac Newton Institute for Mathematical
Sciences in Cambridge, UK, during the 2011 programme on the Design and
Analysis of Experiments. This research was partially 
supported by the National Science Center Grant 
DEC-2011/01/B/ST1/01413
(K. Filipiak and A. Markiewicz) and by the Collaborative Research
Center Statistical modeling of nonlinear dynamic processes
(SFB 823, Teilprojekt C2) of the German Research Foundation (J. Kunert).
Part of the work was done while R. A. Bailey and P. J. Cameron held Hood Fellowships at the
University of Auckland in 2014.

\bigskip



\end{document}